\documentclass[12pt]{article}
\usepackage{amssymb, amsmath, amsthm, amsfonts, enumerate}
\usepackage{amsmath,setspace,scalefnt}
\usepackage[usenames,dvipsnames,svgnames,table]{xcolor}
\usepackage{graphicx,tikz,caption,subcaption}
\usetikzlibrary{patterns}
\usetikzlibrary{shapes}
\usepackage{fullpage}
\usepackage{hyperref}
\usepackage{enumitem,verbatim}

\definecolor{gray}{rgb}{0.25, 0.25, 0.25}

\newtheorem{theorem}{Theorem}[section]
\newtheorem{lemma}[theorem]{Lemma}
\newtheorem{cor}[theorem]{Corollary}

\theoremstyle{definition}

\theoremstyle{plain}
\newtheorem{claim}[theorem]{Claim}

\newtheorem{prop}[theorem]{Proposition}
\theoremstyle{definition}

\theoremstyle{definition}

\theoremstyle{definition}

\theoremstyle{definition}
\newtheorem{defn}[theorem]{Definition}
\theoremstyle{definition}

\theoremstyle{definition}

\theoremstyle{definition}

\newcommand{\ep}{\varepsilon}

\newcommand{\al}{\alpha}
\newcommand{\de}{\delta}

\newcommand{\De}{\Delta}

\newcommand{\Ho}{\H{o}}

\newcommand{\Int}{\mathrm{Int}}
\newcommand{\Ext}{\mathrm{Ext}}

\newcommand{\imp}{\Rightarrow}

\title{A proof of Mader's conjecture on large clique subdivisions in $C_4$-free graphs}
 \author{
Hong Liu
\thanks{Mathematics Institute and DIMAP, University of Warwick, Coventry, CV4 7AL, UK.  Email: {\tt h.liu.9@warwick.ac.uk}. Supported by ERC grant 306493 and EPSRC grant EP/K012045/1.
}
 \quad\quad
 Richard Montgomery
 \thanks{Trinity College, Cambridge, CB2 1TQ, UK. Email: {\tt r.h.montgomery@dpmms.cam.ac.uk}. The research leading to these results was partially supported by the European Research Council
under the European Union's Seventh Framework Programme (FP/2007--2013) / ERC Grant
Agreement no.\ 258345.}
 }

\begin{document}
\maketitle

\begin{abstract} Given any integers $s,t\geq 2$, we show there exists some $c=c(s,t)>0$ such that any $K_{s,t}$-free graph with average degree $d$ contains a subdivision of a clique with at least $cd^{\frac{1}{2}\frac{s}{s-1}}$ vertices.
In particular, when $s=2$ this resolves in a strong sense the conjecture of Mader in 1999 that every $C_4$-free graph has a subdivision of a clique with order linear in the average degree of the original graph. In general, the widely conjectured asymptotic behaviour of the extremal density of $K_{s,t}$-free graphs suggests our result is tight up to the constant $c(s,t)$.
\end{abstract}


\section{Introduction}
Given a graph $H$, a \emph{subdivision} of $H$ is a graph obtained from $H$ by subdividing each of its edges into internally vertex-disjoint paths. A graph~$G$ contains an $H$-subdivision if~$G$ contains a subdivision of $H$ as a subgraph. Subdivisions have proved a key notion in the connections between graph theory and topology. Indeed, perhaps the most important historical result in topological graph theory is Kuratowski's theorem from 1930 that planar graphs are exactly those graphs which do not contain a subdivision of the complete graph with five vertices or a subdivision of the complete bipartite graph with three vertices in each class~\cite{Kur}.

In 1967, Mader~\cite{M-ori} proved the fundamental extremal result that, for each integer $d\ge 1$, there is some $c>0$ such that any graph $G$ with average degree $d(G)$ at least $c$ contains a subdivision of the complete graph with $d$ vertices, $K_d$. That is, we may define for each $d\ge 1$
\[
s(d)=\inf\{c:d(G)\geq c\implies G\text{ contains a }K_d\text{-subdivision}\}.
\]
Mader~\cite{M-ori}, and independently Erd\Ho s and Hajnal~\cite{E-H}, conjectured that $s(d)=O(d^2)$. This conjecture matched the known lower bound for $s(d)$, as, for example, the disjoint union of complete regular bipartite subgraphs demonstrates that $s(d)\geq d^2/8$ for $d\ge 3$ (as first observed by Jung~\cite{J}).

In 1972, Mader~\cite{M-ori-2} showed that $s(d)=O(2^d)$, but it was not until 1994 that further progress was made by Koml\'os and Szemer\'edi~\cite{K-Sz-1}, who showed that $s(d)=O(d^2\log^\eta d)$, for any fixed $\eta>14$. Shortly afterwards, Bollob\'as and Thomason~\cite{B-Th} finally confirmed that $s(d)=O(d^2)$, before Koml\'os and Szemer\'edi~\cite{K-Sz-2} were able to improve their own methods to give an independent proof. These methods involved graph expansion and have formed the basis for many constructions introduced both here and elsewhere (e.g.~\cite{BLSh,K-O-1,K-O-2,K-O-3,Richard}).
Currently, it is known that
\begin{equation}\label{bounds}
(1+o(1))9d^2/64\leq s(d)\leq (1+o(1))10d^2/23,
\end{equation}
where the upper bound is due to K\"uhn and Osthus~\cite{KO-conn} and the lower bound is due to an example by \L uczak.


The extremal examples used to prove the lower bounds mentioned above consist of the disjoint union of dense bipartite graphs. Mader~\cite{M} conjectured that any $C_4$-free graph contains a subdivision of a clique with order linear in its average degree.
Towards this, K\"uhn and Osthus~\cite{K-O-1,K-O-3} proved that if a graph has sufficiently large girth, then it contains a subdivision of a clique with order linear in its average degree. In fact, they deduced this from the stronger result that such a graph, $G$ say, contains a subdivision of a clique with $\de(G)+1$ vertices.
Recently, Balogh, Sharifzadeh and the first author~\cite{BLSh} showed that for any fixed $k\geq 3$, each $C_{2k}$-free graph contains a subdivision of a clique with order linear in its average degree.

Approaching Mader's conjecture from a different direction, K\"uhn and Osthus~\cite{K-O-2} showed that each $C_4$-free graph $G$ contains a subdivision of a clique with at least $d(G)/\log^{12}(d(G))$ vertices, when $d(G)$ is sufficiently large. In fact, they were able to show that for all integers $s,t \ge 2$ there exists a $d_0=d_0(s,t)$ such that
every $K_{s,t}$-free graph $G$ of average degree $d\geq d_0$ contains a subdivision of a clique with at least $d^{\frac{1}{2}\frac{s}{s-1}}/\log^{12}d$ vertices~\cite{K-O-2}. For each $s$ and $t$, this is likely to be tight up to the logarithmic term (as discussed below, and in~\cite[Section 4]{K-O-2}). Inspired by~\cite{BLSh} and~\cite{Richard}, we will give new constructions of clique subdivisions to show the following improvement.




\begin{theorem}\label{mainC4} For all integers $t\ge s\ge 2$ there exists some constant $c=c(s,t)$ so that the following holds for every $d>0$. Every $K_{s,t}$-free graph $G$ with average degree $d$ contains a subdivision of a clique with at least $cd^{\frac{1}{2}\frac{s}{s-1}}$ vertices.
\end{theorem}

\noindent Note that Theorem~\ref{mainC4} demonstrates not only that Mader's conjecture is true, but in fact for any fixed $t\geq 2$ the exclusion of $K_{2,t}$-subgraphs in a graph is sufficient to guarantee a subdivision of a clique with order linear in the average degree of the graph. 

In~\cite{K-O-2}, K\"uhn and Osthus proved the following proposition.
\begin{prop} \cite[Proposition 13]{K-O-2}
\label{KOprop}
For every $c>0$ and all integers $t\ge s\ge 2$ there is a constant $C=C(c,s,t)$ such that no $K_{s,t}$-free graph $G$ with $e(G)\geq c|G|^{2-1/s}$ contains a subdivision of a complete graph of order at least $C d(G)^{\frac 12\frac{s}{s-1}}$.
\end{prop}
It is widely expected that for each $t\ge s\ge 2$ there exists some $c'=c'(s,t)>0$ for which there are arbitrarily large $K_{s,t}$-free graphs $G$ with at least $c'|G|^{2-1/s}$ edges. If this is true, then Proposition~\ref{KOprop} implies that Theorem~\ref{mainC4} is tight up to the constant $c(s,t)$. In particular, such graphs are known to exist when $s=2,3$ (see e.g.~\cite{FS}) and when $t\ge (s-1)!+1$ (see~\cite{ARS}). Thus, in these cases Theorem~\ref{mainC4} is tight up to the constant $c(s,t)$. Due to K\H{o}v\'ari, S\'os and Tur\'an~\cite{K-S-T}, it has long been known that there exists a constant $C'=C'(s,t)$ for which every graph $G$ with more than $C'|G|^{2-1/s}$ edges contains a copy of $K_{s,t}$, and we use a key lemma leading to this result in our proof (Lemma~\ref{lem-KST}).

Lastly, we would like to discuss whether the conditions we consider (that is, the exclusion of some fixed bipartite graph) are the most natural in forcing a large clique subdivision. For example, considering the extremal example for the upper bound in~\eqref{bounds}, one might ask whether a graph with no small subgraphs which are almost as dense as the parent graph must contain a large clique subdivision. More precisely, for any graph $G$, let $f(G)$ be the number of vertices in the smallest subgraph $H$ of $G$ with $d(H)\geq d(G)/100$. It would be possible to use similar constructions to those given here to show that there exists some constant $c>0$ such that every graph $G$ contains a subdivision of a clique with at least $c\cdot \min\{\sqrt{f(G)/\log (f(G))},d(G)\}$ vertices. We had hoped to improve this by removing the logarithmic term, thus generalising both Theorem~\ref{mainC4} and the result of Bollob\'as and Thomason~\cite{B-Th}, and Koml\'os and Szemer\'edi~\cite{K-Sz-2}, that $s(d)=O(d^2)$. This, however, is not possible, as shown by the following counterexample constructed for some large $d$ with $1000|d$.

\medskip

\noindent\textbf{Counterexample:}
Take a random $1000$-regular graph $H$ with $d$ vertices and blow up each vertex by an empty clique with $d/1000$ vertices. Let the resulting graph be $G$.

\medskip


If $G$ contains a subdivision of a clique with $t$ core vertices in one of the empty cliques, then, considering the paths between these core vertices and the size of their neighbourhood, we must have that $\binom{t}{2}\leq d$. Therefore, any clique subdivision in $G$ has at most $\sqrt{2d}$ core vertices in the blowup of each vertex in~$H$. There are at most (say) $d^{1/3}$ vertices within a distance $(\log d)/100$ of any single vertex in~$H$. Therefore, if a clique subdivision in $G$ has more than $4d^{5/6}$ core vertices then at least half of the subdivided paths must contain at least $(\log d)/100$ vertices. If $G$ contains a $K_t$-subdivision with $t\geq 4d^{5/6}$, then, as the paths in a $K_t$-subdivision are internally vertex-disjoint, $\binom t2/2\leq 200|G|/\log d $. As $G$ has $d^2/1000$ vertices, it therefore has no clique subdivision with more than $d/\sqrt{\log d}$ core vertices. Finally, note that with high probability we expect $f(G)=\Omega(|G|)=\Omega(d^2)$, demonstrating that for any fixed $c>0$ a graph $G$ may contain no subdivision of a clique with at least $c\cdot \min\{\sqrt{f(G)},d(G)\}$ vertices.

\medskip

In Section~\ref{sec-prelim} we introduce the results from the literature we need for our proof, before dividing the proof of Theorem~\ref{mainC4} into three cases and giving an overview of the rest of the paper.

\medskip

\noindent\textbf{Notation:} Given a graph $G$, denote its average degree by $d(G)$. For a set of vertices $X\subseteq V(G)$, denote its \emph{external neighbourhood} by $N(X):=\{u\not\in X: uv\in E(G) \mbox{ for some } v\in X\}$. Furthermore, set $N^{1}_G(X):=N(X)$ and for each $i\ge 1$, define $N^{i+1}_G(X):=N(N^{i}_G(X))$ iteratively. Denote by $B_G^r(X)$ the ball of radius $r$ around $X$, i.e.~$B_G^r(X)=\cup_{i\le r}N_G^i(X)$. Let $B_G(X)=B_G^1(X)$, and for each $r$ and vertex $v$, let $N^{r}_G(\{v\})=N^r_G(v)$ and $B^r_G(\{v\})=B^r_G(v)$. We omit the index $G$ if the underlying graph is clear from context. For $k\in\mathbb{N}$, denote by $X^{(k)}$ the family of all $k$-sets in $X$. A vertex in an $H$-subdivision is a \emph{core} vertex if it corresponds to a vertex in $H$, i.e.~it is not an internal vertex in any of the vertex-disjoint paths corresponding to the edges in $H$.

\medskip

We will omit floor and ceiling signs when they are not crucial.

\section{Preliminaries}\label{sec-prelim}

\subsection{Graph expansion}
We need the following notion of graph expansion, which was introduced by Koml\'os and Szemer\'edi~\cite{K-Sz-1}. For $\ep_1>0$ and $k>0$, we first let $\ep(x)$ be
the function
\begin{eqnarray}\label{epsilon}
\ep(x)=\ep(x,\ep_1,k):=\left\{\begin{tabular}{ l c r }
 $0$ & \mbox{ if } $x<k/5$ \\
 $\ep_1/\log^2(15x/k)$ & $\mbox{ if } x\ge k/5$, \\
\end{tabular}
\right.
\end{eqnarray}
\noindent where, when it is clear from context we will not write the dependency on $\ep_1$ and $k$ of $\ep(x)$. Note that $\ep(x)\cdot x$ is increasing for $x\ge k/2$. 

\begin{defn}
\noindent\textbf{$(\ep_1,k)$-expander:} A graph $G$ is an \emph{$(\ep_1,k)$-expander} if $|N(X)|\ge \ep(|X|)\cdot |X|$ for all subsets $X\subseteq V(G)$ of size $k/2\le |X|\le |V(G)|/2$.
\end{defn}

Koml\'{o}s and Szemer\'{e}di~\cite{K-Sz-1,K-Sz-2} showed that every graph $G$ contains an $(\ep_1,k)$-expander subgraph $H$ that is almost as dense as $G$. 
\begin{theorem}\label{k-sz-expander}
Let $k>0$ and choose $\ep_1>0$ sufficiently small (independently of $k$) so that $\ep(x)=\ep(x,\ep_1,k)$
defined in~\eqref{epsilon} satisfies $\int^\infty_1\frac{\ep(x)}{x}dx<\frac18$. Then every graph $G$ has an $(\ep_1,k)$-expander subgraph $H$ with $d(H)\geq d(G)/2$ and $\delta(H)\geq d(H)/2$.
\end{theorem}
\noindent
Note that the subgraph $H$ in Theorem~\ref{k-sz-expander} might be much smaller than~$G$. For example if~$G$ is a vertex-disjoint collection of many small cliques, then $H$ could be one of those cliques. 

We can find a relatively short path between two sufficiently large sets in an $(\ep_1,k)$-expander, even after the deletion of an arbitrary, but smaller, set of vertices. This is formally captured in the following result (see Corollary 2.3 in~\cite{K-Sz-2}).
\begin{lemma}\label{diameter}
If $G$ is an $n$-vertex $(\ep_1,k)$-expander, then any two vertex sets, each of size at least
$x\ge k$, are of distance at most $$diam:=diam(n,\ep_1,k)=\frac{2}{\ep_1}\log^3(15n/k)$$ apart. This remains true even after deleting $x\ep(x)/4$ arbitrary vertices from $G$.
\end{lemma}

\subsection{Bipartite $K_{s,t}$-free graphs}
It will simplify our constructions to work within a bipartite graph. This will be possible due to the following well-known result.

\begin{lemma}\label{bipsub} Within any graph $G$ there is a bipartite subgraph $H$ with $d(H)\geq d(G)/2$.
\end{lemma}

We will use the $K_{s,t}$-free property of our graphs primarily through the following result of K\H{o}v\'ari, S\'os and Tur\'an~\cite{K-S-T} (see also~\cite[IV, Lemma 9]{bollobook}).

\begin{lemma}\label{lem-KST}
Let $G=(A,B)$ be a bipartite graph that does not contain a copy of $K_{s,t}$ with~$t$ vertices in $A$ and $s$ vertices in $B$. Then
$$|A|{d(A)\choose s}\le t{|B|\choose s},$$
where $d(A)=\sum_{v\in A}\frac{d(v)}{|A|}$ is the average degree in $G$ of the vertices in $A$.
\end{lemma}
We will use this lemma mainly through the following corollary.
\begin{cor}\label{cor-KST} Let $G=(A,B)$ be a bipartite graph that does not contain a copy of $K_{s,t}$ with~$t$ vertices in $A$ and $s$ vertices in $B$, and in which every vertex in $A$ has at least $\delta$ neighbours in $B$. Then, $|B|\geq \delta|A|^{1/s}/et$.
\end{cor}
\begin{proof}
As $d(A)\geq \delta$, using Lemma~\ref{lem-KST} we have
\[
|A|\left(\frac{\de}{s}\right)^s\leq |A|{d(A)\choose s}\le t{|B|\choose s}\leq t\left(\frac{e|B|}{s}\right)^s\leq \left(\frac{et|B|}{s}\right)^s.
\]
Taking an appropriate root and rearranging gives the required inequality.
\end{proof}
We also use more directly the following version of K\H{o}v\'ari, S\'os and Tur\'an's theorem~\cite{K-S-T}.
\begin{theorem}\label{KST-thm}
For each $s,t\geq 2$, and every $K_{s,t}$-free graph~$G$, we have $2t|G|\geq (d(G))^{s/s-1}$.
\end{theorem}

\subsection{Division of the proof of Theorem~\ref{mainC4} into cases}\label{subsec-div}
We will divide the proof of Theorem~\ref{mainC4} into three main lemmas. Using Theorem~\ref{k-sz-expander}, we will find within our graph $G$ a subgraph which is almost as dense as $G$ but also has some expansion properties. The first main lemma, Lemma~\ref{max-deg-reduction}, will either find the required subdivision or a large dense subgraph which retains some useful expansion properties while additionally having a small maximum degree.

Thus, this reduces the problem to finding a subdivision in a graph with a certain expansion property and a small maximum degree. The construction we use differs according to the density of the subgraph. The dense case is covered by Lemma~\ref{lem-dense}, while the sparse case is covered by Lemma~\ref{sparsesub}.

The first of these lemmas is adapted and generalised from a lemma in~\cite{BLSh}.

\begin{lemma}\label{max-deg-reduction}
For any $0<\ep_1<1$, and integers $t\ge s\ge 2$, there exists $\ep_2':=\ep_2'(\ep_1,t)>0$ such that for any $0<\ep_2\le\ep_2'$ and $K\ge 100/\ep_2$, there exists some $c_0:=c_0(\ep_1,\ep_2,s, K)>0$ for which the following holds for any $d$. Let $G$ be a bipartite, $K_{s,t}$-free, $(\ep_1,\ep_2d^{s/(s-1)})$-expander with $\de(G)\ge d/8$. Then either $G$ contains a subdivision of a clique of order $c_0d^{\frac{1}{2}\frac{s}{s-1}}$ or a subgraph $H$ with $\de(H)\ge d/16$, $|H|\geq Kd^{s/(s-1)}$ and $\Delta(H)\le d\log^{10s}(|H|/d^{s/(s-1)})$ which is an $(\ep_1/2,\ep_2d^{s/(s-1)})$-expander.
\end{lemma}

When the subgraph found using Lemma~\ref{max-deg-reduction} is dense, we will use the following lemma.

\begin{lemma}\label{lem-dense}
Let $0<\ep_1,\ep_2<1$ and let $t\ge s\geq 2$ be integers. Then, for sufficiently large $K:=K(\ep_1,\ep_2,s,t)>0$, the following holds for any integers $n$ and $d$ with $n\geq Kd^{s/(s-1)}$ and $d\ge \log^{20s}n$. If $G$ is an $n$-vertex bipartite, $K_{s,t}$-free, $(\ep_1,\ep_2d^{s/(s-1)})$-expander with $\de(G)\ge d/16$ and $\Delta(G)\le d\log^{10s}(n/d^{s/(s-1)})$, then $G$ contains a $K_{\ell/10^4t}$-subdivision for $\ell=d^{\frac{1}{2}\frac{s}{s-1}}$.
\end{lemma}

When the subgraph found using Lemma~\ref{max-deg-reduction} is sparse, we will in fact find a subdivision of a clique of order linear in the average degree of the subgraph, as shown by the following lemma.

\begin{lemma}\label{sparsesub} Let $0<\ep_1<1$, $0<\ep_2\leq 1/10^5t$, and let $t\ge s\ge 2$ be integers. Then there is some $c_1:=c_1(\ep_1,\ep_2,s,t)>0$ for which the following holds for any integers $n$ and $d$ with $d\le\log^{20s}n$.  If $G$ is an $n$-vertex bipartite, $K_{s,t}$-free, $(\ep_1,\ep_2d^{s/(s-1)})$-expander with $\de(G)\ge d/16$ and $\Delta(G)\le d\log^{10s}n$, then $G$ contains a $K_{c_1d}$-subdivision.
\end{lemma}

Due to the number of different constants involved, we will carefully show, as follows, that Theorem~\ref{mainC4} follows from Lemmas~\ref{max-deg-reduction},~\ref{lem-dense} and~\ref{sparsesub}.

\begin{proof}[Proof of Theorem~\ref{mainC4}] 
We first fix all the parameters we need, as follows.
\begin{itemize}
\item Let $\ep_0$ be a constant with the property in Theorem~\ref{k-sz-expander}, and $\ep_1=\min\{\ep_0/2,1/8\}$.
\item Let $t\ge s\ge 2$ be integers.
\item Let $\ep_2=\min\{\ep_2'(2\ep_1,t),1/10^5t\}$, where $\ep_2'$ is a constant with the property in Lemma~\ref{max-deg-reduction}.
\item Let $K'=\max\{K(\ep_1,\ep_2,s,t),100/\ep_2\}$, where $K$ is a constant with the property in Lemma~\ref{lem-dense}.
\item Finally, let $c=\min\{1/10^5t,c_0(2\ep_1,\ep_2,s,K'),c_1(\ep_1,\ep_2,s,t)\}$, where $c_0$ and $c_1$ are constants with the properties in Lemmas~\ref{max-deg-reduction} and~\ref{sparsesub} respectively.
\end{itemize}
\noindent
We will prove Theorem~\ref{mainC4} with the constant $c$.
Let $G$ then be a $K_{s,t}$-free graph with average degree $d>0$, and let $\ell=d^{\frac12\frac{s}{s-1}}$. We seek a $K_{c\ell}$-subdivision in $G$.

By Lemma~\ref{bipsub}, $G$ contains a bipartite subgraph $G_1$ with $d(G_1)\geq d/2$. By the choice of $2\ep_1\leq \ep_0$ with Theorem \ref{k-sz-expander}, there is a subgraph $G_2$ of $G_1$ with $\delta(G_2)\geq d(G_2)/2\ge d(G_1)/4\geq d/8$ which is a $(2\ep_1,\ep_2d^{s/(s-1)})$-expander.

By Lemma~\ref{max-deg-reduction}, if $G_2$ does not contain a $K_{c\ell}$-subdivision, then we can find a subgraph~$G_3$ of $G_2$ with $\delta(G_3)\geq d/16$, $|G_3|\geq K'd^{s/(s-1)}$ and $\Delta(G_3)\leq d\log^{10s}(|G_3|/d^{s/(s-1)})$ which is an $(\ep_1,\ep_2d^{s/(s-1)})$-expander.
If $d\geq \log^{20s}|G_3|$, then by the choice of $\ep_2$ and $K'$ with  Lemma~\ref{lem-dense}, $G_3$ contains a $K_{c\ell}$-subdivision. If $d\leq \log^{20s}|G_3|$, then by the choice of $\ep_2$ and~$c$ with Lemma~\ref{sparsesub}, $G_3$ contains a $K_{c\ell}$-subdivision. Therefore, in all cases, $G$ contains a $K_{c\ell}$-subdivision.
\end{proof}

Our remaining task is therefore to prove Lemmas~\ref{max-deg-reduction},~\ref{lem-dense} and~\ref{sparsesub}, which we do in Sections~\ref{sec-max-deg-reduction},~\ref{sec-dense} and~\ref{sec-sparse} respectively. In the rest of this section we will sketch the constructions we will use. In Section~\ref{sec-conclude} we will make some concluding remarks.

\subsection{Sketch of constructions}
Here we sketch an overview of the three constructions we will use to prove Lemmas~\ref{max-deg-reduction},~\ref{lem-dense} and~\ref{sparsesub}. For simplicity, we will assume that $s=t=2$, and say we wish to construct a $K_{cd}$-subdivision in an $n$-vertex $C_4$-free $(\ep_1,\ep_2d^2)$-expander~$G$ with minimum degree $d$, for some small constants $c,\ep_1$ and $\ep_2$. The main variables we have are~$n$ and~$d$, but we also consider the following important variable depending on them:
\[
diam:=(2/\ep_1)\log^3(15n/\ep_2d^2).
\]
By Lemma~\ref{diameter}, in our graph $G$ we will be able to find a path with length at most $diam$ between any two vertex sets $A$ and $B$ if $|A|,|B|\geq \ep_2d^2$. Furthermore, this will remain true if we delete any vertex set $W$ from the graph, as long as $|A|,|B|\geq |W|\cdot diam$. In particular, we will use that if we have found paths using at most $8c^2d^2diam$ vertices, then we can avoid these vertices while finding a new path with length at most $diam$ between any two vertex sets with size at least $d^2(diam)^2$.

We will take in turn the additional assumptions that i) $G$ has many vertices with degree at least $d(diam)^{3}$, ii) $G$ has maximum degree at most $d(diam)^{3}$, $d\geq \log^{40}n$ and $n\geq d^2(diam)^{8}$, and iii) $G$ has maximum degree at most $d(diam)^{3}$ and $d\leq \log^{40}n$. These assumptions take us close to the conditions of Lemmas~\ref{max-deg-reduction},~\ref{lem-dense} and~\ref{sparsesub} respectively, and permit us to sketch the constructions in a slightly simplified situation. 

\medskip

\noindent\textbf{i) Lemma~\ref{max-deg-reduction} construction.} We assume that $G$ has many vertices with degree at least $d(diam)^3$. This will allow us to find $2cd$ disjoint $4d (diam)^2$-stars in $G$. We choose the root vertices of these stars as our candidate core vertices. We then greedily connect as many different pairs of potential core vertices as possible by paths with length at most $4diam$ whose internal vertices do not contain any potential core vertices and are disjoint between paths. When this process is finished there will trivially be at most $8c^2d^2 diam$ vertices in the paths found. Therefore, there will be at least $cd$ \emph{nice} core vertices $v_i$ among the candidate core vertices which have a set $U_i$ of at least $2d(diam)^2$ neighbours not in any of the paths we found (as the original stars were disjoint). The graph~$G$ has minimum degree at least~$d$, and no two vertices in $U_i$ share a neighbour other than $v_i$ (as $G$ is $C_4$-free). Therefore, $U_i\subset N(v_i)$ will have a neighbourhood of size at least $d^2(diam)^2$, large enough to find a path connecting it to any other such set $U_j$ while avoiding all the vertices in previously found paths (by Lemma~\ref{diameter}). Therefore, in the process we must have found a path between any two nice core vertices, which gives a $K_{cd}$-subdivision, as required.

\medskip

\noindent\textbf{ii) Lemma~\ref{lem-dense} construction.} We assume that $G$ has maximum degree at most $d(diam)^{3}$, $d\geq \log^{40}n$ and $n\geq d^2(diam)^{8}$. The main consequence of the low maximum degree is that it allows us to iteratively find structures with up to $d^2(diam)^{4}$ vertices in total. Indeed, $G$ has at least $nd/2$ edges, and $n/(diam)^{4}$ vertices are incident to at most $nd/diam$ edges. Thus, deleting at most $n/(diam)^{4}\geq d^2(diam)^4$ vertices gives a subgraph almost as dense as $G$ in which we can find further structure. In the following sketch we describe the construction of $2cd$ \emph{units} which each have at most $d(diam)^4$ vertices.

Broadly speaking, our construction for Lemma~\ref{lem-dense} is similar to that used for Lemma~\ref{max-deg-reduction}, but with each potential core vertex replaced disjointly by a structure known as a `unit'. Instead of a single vertex which has $4d(diam)^2$ neighbours, we use $2c d$ \emph{connecting} vertices each with their own $4(diam)^3$ \emph{assigned} neighbours, along with disjoint paths with length at most $2diam$ linking the connecting vertices back to a potential core vertex $v_i$. This structure will further have the property that any set of at least $cd$ connecting vertices and at least $2(diam)^3$ of each of their assigned neighbours have together at least $d^2(diam)^2$ neighbours in the graph.

We will then expand and connect up different pairs of these units using disjoint paths which go to a connecting vertex and then down to the corresponding core vertex. By averaging, when this is done there will be at least $cd$ \emph{nice} units where at least $cd$ connecting vertices, their paths, and at least $2(diam)^3$ of their assigned neighbours, are untouched by the paths we found. By expanding the untouched connecting vertices of two nice units we would be able to find disjointly another path between them, so we must have found a path between any two nice units. This will give a $K_{cd}$-subdivision, as required.

\medskip

\noindent\textbf{iii) Lemma~\ref{sparsesub} construction.} We assume that~$G$ has maximum degree at most $d(diam)^{3}$ and $d\leq \log^{40}n$. 
We start with core vertices that are chosen to be a large graph distance apart in $G$, and greedily find internally vertex-disjoint paths between the core vertices under certain conditions. When we go to find a new path between two core vertices, $v$ and $w$ say, we begin by expanding around $v$ (and similarly around $w$) in three stages - out to distance $(\log\log n)^5$, $\log n/100\log\log n$ and $diam$ respectively. At first we consider successive neighbourhoods expanding out from $v$ while avoiding the vertices in the paths we have found which connect other core vertices to $v$. In this expansion we pick up enough vertices that in the second stage we can expand while avoiding all the vertices in the paths we have found so far (such vertices will not previously be encountered as each path will not come near core vertices which are not one of its endpoints). After this expansion we will have picked up enough vertices that we can avoid all these vertices as well as any other vertices that are close to the other core vertices (such vertices will not previously be encountered as the other core vertices are far from $v$). After performing a similar expansion around $w$ this allows us to find a short path connecting $v$ to $w$ that does not come close to the other core vertices. In this manner, we will be able to connect up all the different pairs of core vertices and find a $K_{cd}$-subdivision.

In all this, the maximum degree condition will be critical - in particular it both guarantees that there are not too many vertices close to each core vertex, so that such vertices can be avoided in the third expansion, and guarantees that there exist enough vertices which are pairwise far enough apart to use as core vertices.




\section{Constructing subdivisions when many vertices have high degree}\label{sec-max-deg-reduction}
In this section, we prove Lemma~\ref{max-deg-reduction}. That is, in a $K_{s,t}$-free graph~$G$ with a certain expansion condition and a minimum degree condition, we either find the subdivision we seek or we may assume an additional maximum degree condition. Essentially, if there are few vertices of high degree, then we can delete these vertices and obtain a subgraph of $G$ with almost the same expansion condition and minimum degree condition, but which additionally has a small maximum degree (see Claim~\ref{claim-1}). If to the contrary there are many vertices of high degree, we can then construct the desired subdivision using some of these high-degree vertices as core vertices. In this construction, we first choose appropriate sets $S_1(v)$ of neighbours of some \emph{selected} high-degree vertices $v$ (see Claim~\ref{claim-bigstars}), before using that $G$ is $K_{s,t}$-free to conclude that $N(S_1(v))$ is large, even when some of the vertices in $S_1(v)$ are deleted. We then connect one-by-one the pairs of the high-degree vertices $v$ through the matching sets $N(S_1(v))$ using a short path (which will exist by Lemma~\ref{diameter}), deleting any vertices we use from other sets $S_1(w)$ to ensure the paths we find are internally disjoint (see Algorithm~P). We do not expect to be able to connect all the pairs of selected high-degree vertices in this manner, but we will show that at the end of the process most of the sets $S_1(v)$ will still be large (see Claim~\ref{claim-2}) and that we will have found a path between any pair of vertices $v$ with a large matching set $S_1(v)$ (see Claim~\ref{claim-3}). These paths will form the desired subdivision.

\begin{proof}[Proof of Lemma~\ref{max-deg-reduction}]
Given $0<\ep_1<1$ and integers $t\ge s\ge 2$, take
\begin{equation}\label{ep2-defn}
\ep_2'=\ep_2'(\ep_1,t):=\frac{1}{10t}\min\left\{e^{-100et},e^{-100/\ep_1}\right\}.
\end{equation}
Given further constants $0<\ep_2\le \ep_2'$ and $K\ge 100/\ep_2$, let 
\begin{equation}\label{d0ep3defn}
d_0=d_0(\ep_1,\ep_2,s):=\left(\frac{100}{\ep_1\ep_2}\right)^{s}\text{ and }c_0=c_0(\ep_1,\ep_2,s,K):=\min\left\{\frac{1}{10\log^{10s}(30K/\ep_2)},\frac{1}{d_0}\right\}.
\end{equation}
We will show that the lemma holds with $c_0$.
Let $d>0$ and $n\in \mathbb{N}$. Let $G$ be an $n$-vertex bipartite $K_{s,t}$-free $(\ep_1,\ep_2d^{s/s-1})$-expander graph with $\delta(G)\ge d/8$. Note that if $d\leq d_0$, then, as $G$ contains a $K_1$-subdivision as $\de(G)\geq d/8>0$, and $c_0d^{\frac12\frac{s}{s-1}}\leq 1$, we have the required subdivision. We can thus assume that $d\geq d_0$.

Let
\begin{equation}\label{Deltadefn}
\Delta:=\max\left\{d/8,\quad c_0d\left(\log\frac{15n}{\ep_2d^{s/(s-1)}}\right)^{10s}\right\},
\end{equation}
and let $L\subseteq V(G)$ be the set of vertices in~$G$ with degree at least $\Delta$.

\begin{claim}\label{claim-1} If $|L|\leq d/16$, then $H:=G-L$ is an $(\ep_1/2,\ep_2d^{s/(s-1)})$-expander satisfying $\de(H)\ge d/16$, $|H|\geq Kd^{s/(s-1)}$ and $\De(H)\le d\log^{10s}(|H|/d^{s/(s-1)})$.
\end{claim}
\begin{proof}[Proof of Claim~\ref{claim-1}] Suppose that $|L|\leq d/16$. As $|L|<\delta(G)<|G|$, we know $L\neq V(G)$, and therefore, by the definition of $L$, $\Delta>\delta(G)\geq d/8$. Thus, from~\eqref{Deltadefn}, we have
\begin{eqnarray}
\De=c_0d\left(\log\frac{15n}{\ep_2d^{s/(s-1)}}\right)^{10s}\geq \frac{d}{8} \quad &\imp& 
\quad \left(\log\frac{15n}{\ep_2d^{s/(s-1)}}\right)^{10s}\ge \frac{1}{8c_0}\overset{\eqref{d0ep3defn}}{\ge} \left(\log\frac{30K}{\ep_2}\right)^{10s}\nonumber \\
&\imp& \quad n\geq 2Kd^{s/(s-1)}.\label{useful}
\end{eqnarray} 
As $n> \delta(G)\geq d/8$, we have $|H|\ge n-|L|\geq n-d/16\ge n/2$, and thus $|H|\geq Kd^{s/(s-1)}$ by~\eqref{useful}, as required. Furthermore, $\de(H)\ge\de(G)-|L|\ge d/16$.

Using that $|H|/d^{s/(s-1)}\ge K\ge 100/\ep_2$ and $|H|\geq n/2$, the maximum degree of $H$ is at most 
\begin{eqnarray*}
\De&\overset{\eqref{Deltadefn}}{\leq}& c_0d\cdot\log^{10s}\left(\frac{30|H|}{\ep_2d^{s/(s-1)}}\right)\overset{\eqref{d0ep3defn}}{\le} d\cdot\left(\frac{\log\left(\frac{30}{\ep_2}\cdot\frac{|H|}{d^{s/(s-1)}}\right)}{\log(30K/\ep_2)}\right)^{10s}\\
&\le& d\cdot\left(\frac{2\cdot \log\left(\frac{|H|}{d^{s/(s-1)}}\right)}{\log(30K/\ep_2)}\right)^{10s} \overset{\eqref{ep2-defn}}{\leq} d\log^{10s}\left(\frac{|H|}{d^{s/(s-1)}}\right).
\end{eqnarray*}

To finish the proof of the claim it is left to show that $H$ is an $(\ep_1/2,\ep_2d^{s/(s-1)})$-expander.
Set $k:=\ep_2d^{s/(s-1)}$. Since $G$ is an $(\ep_1,k)$-expander and $\ep(x)\cdot x$ is increasing when $x\ge k/2$, for any set $X$ in $H$ of size $x\ge k/2$ with $x\le |H|/2\leq |G|/2$, we have
\begin{eqnarray*}
|N_G(X)|&\ge& x\cdot\ep(x, \ep_1,k)\ge \frac{k}{2}\cdot \ep\left(\frac{k}{2},\ep_1,k\right)= \frac{\ep_2d^{s/(s-1)}}{2}\cdot\frac{\ep_1}{\log^2(15/2)}\\
&\geq& \frac{\ep_2\ep_1}{100}\cdot d^{s/(s-1)}\ge d\ge 2|L|,
\end{eqnarray*}
where $\ep$ is defined in~\eqref{epsilon} and the second last inequality follows as $d\ge d_0=\left(\frac{100}{\ep_1\ep_2}\right)^{s}$. Hence, $|N_{H}(X)|\ge \frac12 |N_G(X)|\ge  \frac{1}{2}x\cdot\ep(x,\ep_1,\ep_2d^{s/(s-1)})=x\cdot\ep(x,\ep_1/2,\ep_2d^{s/(s-1)})$, as required.
\end{proof}

Therefore, if $|L|\leq d/16$ then we have can find a subgraph $H$ that would satisfy the lemma. Thus, we may assume that $|L|\geq d/16$. Let $\ell=c_0d^{\frac{1}{2}\frac{s}{s-1}}$. We will now find in $G$ a $K_{\ell}$-subdivision, completing the proof of the lemma.

By Theorem~\ref{KST-thm}, and since $\de(G)\ge d/8$, we have $n/d^{s/(s-1)}\ge 1/128t$. Using this, let
\begin{equation}\label{eq-mlarge}
m:=\log\frac{15n}{\ep_2d^{s/(s-1)}}\ge \log\frac{15}{128t\cdot\ep_2}\overset{\eqref{ep2-defn}}{\ge} \max\left\{100et, \frac{100}{\ep_1}\right\}.
\end{equation}

As $K\ge 100/\ep_2$,~\eqref{ep2-defn} implies that $c_0\leq 1/64$, so that $|L|\ge d/16\geq 4\ell$. Therefore, as~$G$ is bipartite, we can find a subset $L'\subseteq L$ of $2\ell$ vertices in the same partite set. We will use some vertices from $L'$ as the core vertices of our $K_{\ell}$-subdivision.
\begin{claim}\label{claim-bigstars} For each $v\in L'$, we can pick a subset $S_1(v)\subseteq N(v)$ such that

(i) $|S_1(v)|=\Delta/2$, and

(ii) each vertex $u\in S_1(v)$ is adjacent to at most $c_0d/\ell$ vertices in $L'$.
\end{claim}
\begin{proof}
Fixing $v\in L'$, let $A=\{w\in N(v):|N(w)\cap L'|\ge d^{\frac 12\frac{s-2}{s-1}}+1\}$. Note that $d^{\frac{1}{2}\frac{s-2}{s-1}}=c_0d/\ell$, so that if $w\in N(v)\setminus A$ then, by the definition of $A$, $w$ has at most $c_0d/\ell$ neighbours in $L'$. Now, there is no copy of $K_{s-1,t}$ with $t$ vertices in $A$ and $s-1$ vertices in $L'\setminus \{v\}$ since such a copy of $K_{s-1,t}$ together with $v$ would form a copy of $K_{s,t}$. Therefore, using Corollary~\ref{cor-KST} with $\delta=d^{\frac 12\frac{s-2}{s-1}}$ and $B=L'\setminus \{v\}$, we have
\[
d^{\frac 12\frac{s-2}{s-1}}|A|^{1/(s-1)}/et\leq |L'\setminus\{v\}|\leq 2\ell=2c_0d^{\frac 12\frac{s}{s-1}}.
\]
Hence, using~\eqref{eq-mlarge} and~\eqref{Deltadefn}, we have $|A|\leq d(2c_0et)^{(s-1)}\le c_0^{s-1}dm^s/2\leq \Delta/2$. Thus, $|N(v)\setminus A|\geq \Delta/2$, and we may pick a set $S_1(v)\subset N(v)\setminus A$ with size $\Delta/2$. Picking such a set for each $v\in L'$ satisfies the claim.
\end{proof}

\medskip

\noindent\textbf{Algorithm~P}: Take sets $S_1(v)$, $v\in L'$, with the properties in Claim~\ref{claim-bigstars}.
Connect different pairs of core vertices $v\in L'$ greedily through these sets under the following rules:

\begin{enumerate}[label = \textbf{P\arabic{enumi}}]
\item At each step, build a path of length at most $2m^4$ connecting a new pair of core vertices, avoiding vertices used in previous connections and in $L'$. \label{step1}
\item During the whole process, discard a core vertex $v\in L'$ if more than $\De/4$ vertices in $S_1(v)$ are used in previous connections.\label{step2}
\end{enumerate}

This process will result in a $K_{\ell}$-subdivision, as shown by the following two claims.

\begin{claim}\label{claim-2} At the end of the process at least $\ell$ core vertices remain undiscarded.
\end{claim}
\begin{claim}\label{claim-3} At the end of the process all possible paths have been created between the remaining core vertices.
\end{claim}

\begin{proof}[Proof of Claim~\ref{claim-2}]
By \ref{step1}, at most ${2\ell\choose 2}\cdot 2m^4\le 4\ell^2m^4$ vertices have been used to create connections in the entire process. By~\ref{step2}, if a core vertex $v\in L'$ has been discarded then at least $\De/4$ vertices in $S_1(v)$ have been used in creating connections. Note that by Claim~\ref{claim-bigstars} (ii), each vertex is in at most $c_0d/\ell$ of the sets $S_1(v)$, $v\in L'$. Hence, at the end of process the number of discarded core vertices is at most
$$\frac{4\ell^2m^4\cdot c_0d/\ell}{\De/4}\overset{\eqref{Deltadefn},\eqref{eq-mlarge}}{\le}\frac{16\ell m^4c_0d}{c_0dm^{10s}}\overset{\eqref{eq-mlarge}}{\le} \ell.$$
Thus, at least $|L'|-\ell=\ell$ core vertices remain at the end of the process. 
\end{proof}

\begin{proof}[Proof of Claim~\ref{claim-3}]
Suppose that, having perhaps connected some pairs of core vertices subject to~\ref{step1} and~\ref{step2}, the current pair of undiscarded core vertices to be connected is $\{v,v'\}$. Let $A\subseteq S_1(v)$ be the set of vertices in $S_1(v)$ not used in previous connections. Since $v$ has not been discarded, by~\ref{step2}, $|A|\ge |S_1(v)|-\De/4=\De/4$. Let $B:=N(A)\setminus\{v\}\subseteq N(S_1(v))\setminus\{v\}$ and note that $G[A,B]$ does not contain a copy of $K_{s-1,t}$ with $t$ vertices in $A$ and $s-1$ vertices in $B$. Then, by Corollary~\ref{cor-KST}, we have
\[
|B|\geq \frac{1}{et}\left(\frac{d}{8}-1\right)\left(\frac{\Delta}{4}\right)^{1/(s-1)}\ge \frac{d}{64et}\cdot \De^{1/(s-1)},
\]
and hence, by~\eqref{Deltadefn},~\eqref{eq-mlarge} and~\eqref{ep2-defn}, we have
\begin{equation}\label{Bineqs}
|B|\geq \frac{d}{64et} \cdot c_0d^{\frac{1}{s-1}}m^{10}\geq \ell^2 m^9\quad\text{ and }\quad |B|\geq \frac{d}{64et} \cdot \left(\frac{d}{8}\right)^{\frac{1}{s-1}}\geq 2\ep_2 d^{s/(s-1)}.
\end{equation}
To make connections according to~\ref{step1}, we have to exclude vertices from $L'$ as well as used vertices from $B$. Therefore, the number of available vertices in $B$ is, using~\eqref{Bineqs}, at least 
$$y:=|B|-|L'|-4\ell^2m^4\ge|B|/2\ge \max\{\ell^2m^9/2, \ep_2 d^{s/(s-1)}\}.$$
Letting $k=\ep_2d^{s/(s-1)}$, by~\eqref{eq-mlarge}, we have that $\ep(n,\ep_1,k)=\ep_1/m^2\ge 1/m^3$, where $\ep$ is defined in~\eqref{epsilon}. Since $\ep(x,\ep_1,k)$ is decreasing when $x\ge k$, we have
$$\frac{1}{4}\cdot y\cdot \ep(y,\ep_1,k)\ge \frac{y}{4}\cdot \ep(n,\ep_1,k)\ge \frac{\ell^2m^9/2}{4}\cdot \frac{1}{m^3}>5\ell^2m^4.$$ 
Similarly, if $A'$ is the set of available vertices in $S_1(v')$ and $y'$ is the number of available vertices in $N(A')\setminus\{v'\}$, then $\frac{1}{4}\cdot y'\cdot \ep(y',\ep_1,k)>5\ell^2m^4$. Recall that the number of vertices to avoid in~\ref{step1} is at most $|L'|+4\ell^2m^4\le 5\ell^2m^4$. Thus, by Lemma~\ref{diameter}, there is a path of length at most $2\log^3(15n/\ep_2d^{s/(s-1)})/\ep_1\leq m^4$ between the set of available vertices in $N(A)$ and $N(A')$, which yields a $v,v'$-path of length at most $m^4+4\le 2m^4$, as desired.
\end{proof}
Thus, the process creates a $K_\ell$-subdivision.
\end{proof}


\section{Constructing subdivisions in a dense expander}\label{sec-dense}
In this section, we prove Lemma~\ref{lem-dense}, which covers the case for Theorem~\ref{mainC4} where the subgraph that may be found using Lemma~\ref{max-deg-reduction} is dense. Here this means that $d\ge \log^{20s} n$. We need the following two definitions, which are depicted in Figure~\ref{fig-unit-hub}.

\medskip

\noindent\textbf{$(h_1,h_2)$-Hub.} Given integers $h_1,h_2>0$, an \emph{$(h_1,h_2)$-hub} is a graph consisting of a center vertex $u$, a set $S_1(u)\subseteq N(u)$ of size $h_1$, and pairwise disjoint sets $S_1(z)\subset N(z)$ of size $h_2$ for each $z\in S_1(u)$. Denote by $H(u)$ a hub with center vertex $u$ and write $B_1(u)=\{u\}\cup S_1(u)$ and $S_2(u)=\bigcup_{z\in S_1(u)} S_1(z)$. For any $z\in S_1(u)$, write $B_1(z)=\{z\}\cup S_1(z)$.

\medskip

\noindent\textbf{$(h_0,h_1,h_2,h_3)$-Unit.} Given integers $h_0,h_1,h_2,h_3>0$, an \emph{$(h_0,h_1,h_2,h_3)$-unit}~$F$ is a graph consisting of a core vertex $v$, $h_0$ vertex-disjoint $(h_1,h_2)$-hubs $H(u_1),\ldots,H(u_{h_0})$ and pairwise disjoint $v,u_{j}$-paths of length at most $h_3$. By the \emph{exterior} of the unit, denoted by $\Ext(F)$, we mean $\bigcup_{j=1}^{h_0}S_2(u_{j})$. Denote by $\Int(F):=V(F)\setminus\Ext(F)$ the \emph{interior} of the unit.

\bigskip

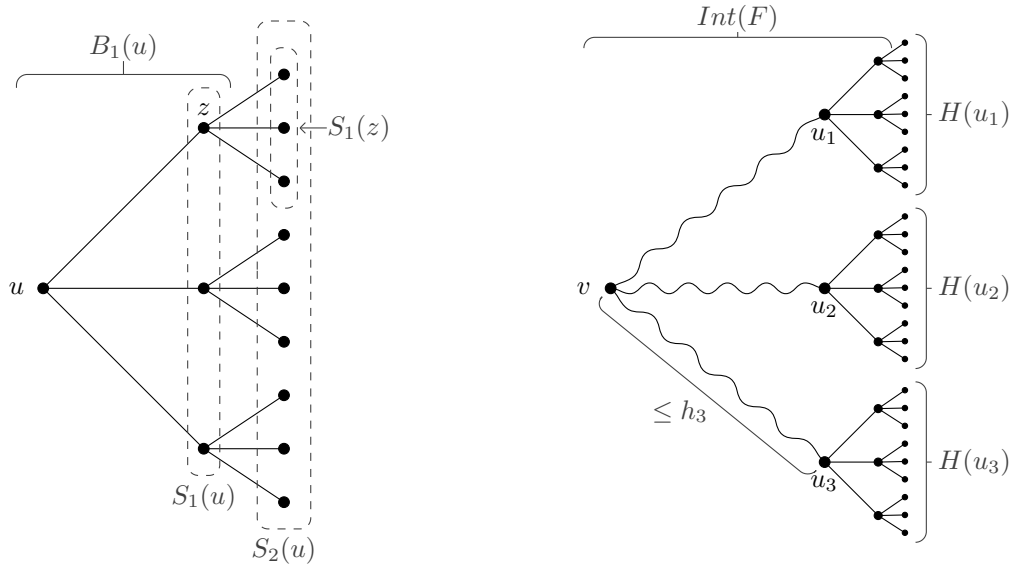
\begin{figure}[t!]
\centering
    \begin{subfigure}[b]{0.45\textwidth}
        \centering
        \resizebox{\linewidth}{!}{
            {\scalefont{0.8}
\begin{tikzpicture}[scale=0.07]

\draw [white] (-10,-10) -- (-10,110) -- (85,110) -- (85,-10) -- (-10,-10);

\draw (5,50) node {$u$};
\draw (40,83.5) node {$z$};
\draw [gray] (40,11) node {$S_1(u)$};
\draw [gray] (55,1) node {$S_2(u)$};
\draw [gray] (25,95) node {$B_1(u)$};

\draw [gray] (58,80) -- (63,80);
\draw [gray] (58,80) -- (59,81);
\draw [gray] (58,80) -- (59,79);

\draw [gray] (69,80) node {$S_1(z)$};

\draw [fill] (10,50) circle [radius=1cm];

\foreach \x in {0,1,2}
{
\draw [fill] (40,{20+\x*30}) circle [radius=1cm];
\draw (10,50)--(40,{20+\x*30});
\foreach \y in {0,1,2}
{
\draw [fill] (55,{10+\x*30+\y*10}) circle [radius=1cm];
\draw (40,{20+\x*30})--(55,{10+\x*30+\y*10});
}
};

\draw [gray, dashed, rounded corners] (37,50) -- (37,87.5) -- (43,87.5) -- (43,15) -- (37,15) -- (37,50);

\draw [gray, dashed, rounded corners] (52.5,80) -- (52.5,95) -- (57.5,95) -- (57.5,65) -- (52.5,65) -- (52.5,80);

\draw [gray, dashed, rounded corners] (50,53.5) -- (50,100) -- (60,100) -- (60,5) -- (50,5) -- (50,53.5);

\draw [gray, rounded corners] (45,87.5) -- (45,90) -- (5,90) -- (5,87.5);
\draw [gray] (25,90) -- (25,92.5);

\end{tikzpicture}
}
        }
    \end{subfigure}\quad\quad
    \begin{subfigure}[b]{0.45\textwidth}
        \centering
        \resizebox{\linewidth}{!}{
            {\scalefont{0.8}
\begin{tikzpicture}[scale=0.07]

\draw [white] (-10,-10) -- (-10,110) -- (85,110) -- (85,-10) -- (-10,-10);

\draw (-5,50) node {$v$};
\draw [gray] (23.75,100.5) node {$Int(F)$};

\draw  (0,50) to [out=0,in=180] 
(3.63636363636364,49) to [out=0,in=180] 
(7.27272727272727,51) to [out=0,in=180] 
(10.9090909090909,49) to [out=0,in=180] 
(14.5454545454545,51) to [out=0,in=180] 
(18.1818181818182,49) to [out=0,in=180] 
(21.8181818181818,51) to [out=0,in=180] 
(25.4545454545455,49) to [out=0,in=180] 
(29.0909090909091,51) to [out=0,in=180] 
(32.7272727272727,49) to [out=0,in=180] 
(36.3636363636364,51) to [out=0,in=180] 
(40,50);

\draw  (0,50) to [out=39.0938588862295,in=219.093858886229] 
(4.2669562614581,52.1784314544292) to [out=39.0938588862295,in=219.093858886229] 
(6.64213464763281,56.6852049092072) to [out=39.0938588862295,in=219.093858886229] 
(11.5396835341854,58.0875223635201) to [out=39.0938588862295,in=219.093858886229] 
(13.9148619203601,62.5942958182981) to [out=39.0938588862295,in=219.093858886229] 
(18.8124108069127,63.996613272611) to [out=39.0938588862295,in=219.093858886229] 
(21.1875891930873,68.503386727389) to [out=39.0938588862295,in=219.093858886229] 
(26.0851380796399,69.9057041817019) to [out=39.0938588862295,in=219.093858886229] 
(28.4603164658146,74.4124776364799) to [out=39.0938588862295,in=219.093858886229] 
(33.3578653523672,75.8147950907928) to [out=39.0938588862295,in=219.093858886229] 
(35.7330437385419,80.3215685455708) to [out=39.0938588862295,in=219.093858886229] 
(40,82.5);

\draw   (0,50) to [out=-39.0938588862295,in=140.906141113771] 
(3.00577101126917,46.2693405453383) to [out=-39.0938588862295,in=140.906141113771] 
(7.90331989782174,44.8670230910254) to [out=-39.0938588862295,in=140.906141113771] 
(10.2784982839964,40.3602496362474) to [out=-39.0938588862295,in=140.906141113771] 
(15.176047170549,38.9579321819344) to [out=-39.0938588862295,in=140.906141113771] 
(17.5512255567237,34.4511587271565) to [out=-39.0938588862295,in=140.906141113771] 
(22.4487744432763,33.0488412728435) to [out=-39.0938588862295,in=140.906141113771] 
(24.823952829451,28.5420678180656) to [out=-39.0938588862295,in=140.906141113771] 
(29.7215017160036,27.1397503637526) to [out=-39.0938588862295,in=140.906141113771] 
(32.0966801021783,22.6329769089746) to [out=-39.0938588862295,in=140.906141113771] 
(36.9942289887308,21.2306594546617) to [out=-39.0938588862295,in=140.906141113771] 
(40,17.5);

\draw [fill] (0,50) circle [radius=1cm];

\foreach \adj in {2.5}
{
  \foreach \x in {0,1,2}
  {
  \draw [fill] (40,{20+\x*30+\x*\adj-\adj}) circle [radius=1cm];
  \draw [gray, rounded corners] (57,{20+\x*30+15+\x*\adj-\adj}) -- (59,{20+\x*30+15+\x*\adj-\adj})-- (59,{20+\x*30-15+\x*\adj-\adj})--(57,{20+\x*30-15+\x*\adj-\adj});
  \draw [gray, rounded corners] (59,{20+\x*30+\x*\adj-\adj}) -- (60,{20+\x*30+\x*\adj-\adj});
  \foreach \y in {0,1,2}
  {
  \draw [fill] (50,{10+\x*30+\y*10+\x*\adj-\adj}) circle [radius=0.75cm];
  \draw (40,{20+\x*30+\x*\adj-\adj})--(50,{10+\x*30+\y*10+\x*\adj-\adj});
  \foreach \z in {0,1,2}
  {
    \draw [fill] (55,{6.7777+\x*30+\y*10+\z*3.3333333+\x*\adj-\adj}) circle [radius=0.5cm];
    \draw (50,{10+\x*30+\y*10+\x*\adj-\adj})--(55,{6.7777+\x*30+\y*10+\z*3.3333333+\x*\adj-\adj});
  };
  };
  };

\foreach \x in {1,2,3}
{
\draw [gray] (68,{110-\x*30-\x*\adj+2*\adj}) node {$H(u_{{\x}})$};
\draw  (39.9,{114-\x*30-\x*\adj+2*\adj-8}) node {$u_{{\x}}$};
};
};


\draw [gray, rounded corners] (52.5,94.944444) -- (52.5,96.944444) -- (-5,96.944444) -- (-5,94.944444);
\draw [gray] (23.75,96.944444) -- (23.75,97.944444);

\draw [gray, rounded corners] (-2,48) -- (-3,47) -- (37,14.5) -- (38,15.5);

\draw [gray] (13,27) node {$\leq h_3$};

\end{tikzpicture}
}
        }
    \end{subfigure}

\vspace{-0.5cm}
\caption{An $(h_1,h_2)$-hub $H(u)$ on the left with $h_1=h_2=3$ and an $(h_0,h_1,h_2,h_3)$-unit $F$ on the right with $h_0=h_1=h_2=3$. The straight lines represent edges, while the wavy lines represent paths (here with length at most $h_3$).} \label{fig-unit-hub}
\centering
\end{figure}

An outline of the proof of Lemma~\ref{lem-dense} is as follows. We need to construct a collection of units whose interiors are pairwise disjoint (see Section~\ref{sec-unit}). We do so iteratively by starting with many vertex-disjoint hubs (found in Section~\ref{sec-hub}), and connecting them (using Algorithm Q in Section~\ref{sec-unit}) in such a way that one of the hubs will be linked to many others (see Claim~\ref{claim-unit}), forming the desired unit. Once this is done, we then connect pairs of units through their exteriors while avoiding all the previously used vertices and the vertices in all these $v,u_j$-paths in the units (see Section~\ref{sec-finish}, and in particular Algorithm R), so that their centre vertices become the core vertices of the required subdivision.

Our use of units is inspired by a similar structure introduced in~\cite{Richard}. Roughly speaking, we face two challenges when attempting to connect some potential core vertices to get a subdivision. Firstly, the paths from each core vertex $v$ must be disjoint, a challenge near~$v$, and secondly the paths must connect $v$ to all the other core vertices. Essentially, through the use of units we compartmentalise these two problems and deal with them separately. In finding the units, we find many paths emerging from a potential core vertex, where the paths end in hubs to help us extend them further. Once we have found the units, we then concern ourselves with extending some of these paths to connect different units until we have formed a subdivision.


\subsection{Constructing hubs}\label{sec-hub}
In order to construct units, we will first find many vertex-disjoint hubs. These are shown to exist by the following lemma.
\begin{lemma}\label{lem-hub}
For each $2\le s\le t$ there is some $d_0$ such that the following holds for each $d\geq d_0$. Let $h_1,h_2\le d^{\frac{1}{2}\frac{s}{s-1}}/400t$ be integers and let $G$ be an $n$-vertex $K_{s,t}$-free bipartite graph with $\delta(G)\geq d/16$. Given any set $W\subseteq V(G)$ with size at most $nd/128 \Delta(G)$, there are in $G-W$ pairwise disjoint $(h_1,h_2)$-hubs of total size at most $nd/128 \Delta(G)$.
\end{lemma}
\begin{proof}
Let $c=1/400t$ and $\ell=d^{\frac{1}{2}\frac{s}{s-1}}$. It suffices to show that given any set $W'\subseteq V(G)$ with size $nd/64 \Delta(G)$, there is a $(c\ell,c\ell)$-hub in $G-W'$. Note that
$$d(G-W')\ge\frac{n\cdot\delta(G)-|W'|\cdot 2\De(G)}{n-|W'|}\ge\frac{nd/16-nd/32}{n}\ge\frac{d}{32}.$$
There is then a subgraph $G'\subseteq G-W'$ with $\de(G')\ge d(G-W')/2\ge d/64$. We will find a $(c\ell,c\ell)$-hub in $G'$.

Let $v$ be an arbitrary vertex in $G'$ and let $A=N(v)$, so that $|A|\geq d/64$. It suffices to find $c\ell$ disjoint stars in $G'-\{v\}$, each with $c\ell$ leaves and its center in $A$, since such stars together with $v$ form a $(c\ell,c\ell)$-hub.
Let $A'\subset A$ be a maximal subset such that we can find $|A'|$ disjoint stars in $G'-\{v\}$, each with $c\ell$ leaves and its center in $A'$. Let $B$ be the union of the leaves of such a collection of stars.

If $|A'|\geq c\ell$, then we are done, so suppose $|A'|<c\ell$, so that $|A\setminus A'|\geq d/64-c\ell\geq d/128$. Each vertex in $A\setminus A'$ has fewer than $c\ell$ neighbours in $V(G')\setminus (B\cup \{v\})$, otherwise that vertex could be added to $A'$ to reach a contradiction. Therefore, as $\de(G')\ge d/64$, each vertex in $A\setminus A'$ has at least $d/64-c\ell-1\geq d/128$ neighbours in $B$.

Note that $G'-\{v\}$ does not contain a copy of $K_{s-1,t}$ with $t$ vertices in $A\setminus A'$ and $s-1$ vertices in $B$, since otherwise such a copy together with $v$ forms a copy of $K_{s,t}$ in $G'$. Therefore, by Lemma~\ref{cor-KST}, we have
\[
|B|\geq (d/128)|A\setminus A'|^{1/(s-1)}/et\geq (d/128)^{s/(s-1)}/et\geq \ell^2/10^5t\ge c^2\ell^2.
\]
As $|B|=|A'|c\ell$, we thus have $|A'|\geq c\ell$, a contradiction.
\end{proof}


\subsection{Constructing units from hubs}\label{sec-unit}
We will now expand hubs to connect them into a unit.
\begin{lemma}\label{lem-unit}
For each $0<\ep_1,\ep_2<1$ and integers $2\le s\le t$, there exists $K:=K(\ep_1,\ep_2,s,t)$ such that the following holds for all $n$ and $d$ with $d\ge \log^{20s}n$ and $n\geq K d^{s/(s-1)}$. Let $c=1/800t$, $\ell=d^{\frac{1}{2}\frac{s}{s-1}}$ and $m=\log^{2s}(n/\ell^2)$, and suppose $G$ is a bipartite $n$-vertex $K_{s,t}$-free $(\ep_1,\ep_2\ell^2)$-expander with $\delta (G)\ge d/16$ and $\Delta(G)\leq dm^5$. Then $G$ contains~$\ell$ $(c\ell, m^2, c\ell,2m)$-units $F_1,\ldots, F_{\ell}$ with core vertices $v_1,\ldots,v_{\ell}$  so that the interiors of all the units $F_i$, that is, the sets $\Int(F_i)$, are pairwise disjoint.
\end{lemma}
\begin{proof} 
We will construct the units iteratively.
Let $W$ be the set of vertices in the interiors of the $\left(c\ell, m^2, c\ell,2m\right)$-units constructed so far. The interior of each such unit has size at most $c\ell\cdot (2m+m^2)\leq 2 c\ell m^2$. Thus, $|W|\le 2c\ell^2m^2$. 

Since $n/\ell^2\ge K$, for sufficiently large $K$, we have that
\begin{eqnarray}\label{eq-mn}
\frac{n}{\ell^2}\ge 128\log^{20s}\left(\frac{n}{\ell^2}\right)=128m^{10}.
\end{eqnarray}
Note that $d\ge \log^{20s}n\ge m^{10}$ and, consequently, since $m\ge \log^{2s}K$, for sufficiently large~$K$,
\begin{eqnarray}\label{eq-ell}
c\ell=cd^{\frac{1}{2}\frac{s}{s-1}}> cd^{1/2}\ge 8m^4.
\end{eqnarray}
Therefore, by Lemma~\ref{lem-hub}, we can find in $G-W$ vertex-disjoint hubs $H(w_1),\ldots, H(w_{m^3})$ and $H(u_1),\ldots, H(u_{\ell m^3})$ such that each $H(w_i)$, $1\le i\le m^3$, is a $(2c\ell, 2c\ell)$-hub and each $H(u_j)$, $1\le j\le \ell m^3$, is a $(2m^2, 2c\ell)$-hub. Indeed, this is possible since $2c\leq 1/400t$, by~\eqref{eq-ell} we have that $2m^2\leq c\ell$,
$$|W|\le 2c\ell^2m^2\overset{\eqref{eq-mn}}{\le} \frac{n}{128m^5}\le \frac{nd}{128\Delta(G)},$$ 
and the total number of vertices in all these hubs is at most 
\[
2(2c\ell)^2\cdot m^3+2(2c\ell)(2m^2)\cdot \ell m^3\le \ell^2 m^5\overset{\eqref{eq-mn}}{\le} \frac{nd}{128\Delta(G)}.
\]

We will construct a unit using some vertex $w_i$ as the core vertex and some subgraphs of the hubs $H(u_j)$ as the hubs in the unit. 

\medskip

\noindent\textbf{Algorithm~Q}: Greedily, connect as many different pairs of vertices $\{w_i,u_j\}$ as possible under the following rules.

\begin{enumerate}[label = \textbf{Q\arabic{enumi}}]
\item \label{STEP1} Each connection uses a path of length at most $2m$, avoiding vertices used in previous connections.
\item \label{STEP2} Each connection avoids using vertices in any set $B_1(w_{i'})$ or $B_1(u_{j'})$, except for at most two vertices each in $B_1(w_i)$ and $B_1(u_j)$ when $\{w_i,u_j\}$ is the pair of vertices being connected.
\end{enumerate}

\begin{claim}\label{claim-unit}
There exists a vertex $w_i$ connected to at least $c\ell$ vertices $u_j$.
\end{claim}

\begin{proof}[Proof of Claim~\ref{claim-unit}]
Suppose to the contrary that each vertex $w_i$ is connected to fewer than $c\ell$ vertices $u_j$ at the end of the process. Let $P$ be the set of all interior vertices in all the connections. Then, by~\ref{STEP1}, 
\begin{equation}\label{laterone}
|P|\le 2m \cdot m^3\cdot c\ell= 2m^4c\ell.
\end{equation}
Let $W'$ contain the vertices in $P$ as well as all the vertices in each set $B_1(u_j)$ if $u_j$ has been connected to at least one of the vertices $w_i$. As there are at most $m^3\cdot c\ell$ such vertices $u_j$, using~\eqref{eq-ell} we have 
\begin{equation}\label{Wupbound}
|W'|\leq |P|+m^3c\ell\cdot (2m^2+1)\overset{\eqref{laterone}}{\le} 2m^4c\ell  +  4m^5c\ell \overset{\eqref{eq-ell}}{\le} c^2\ell^2m.
\end{equation}

For each $1\leq i\leq m^3$, let $T_i=B_1(w_i)\setminus W'$, so that, by~\ref{STEP2}, $|T_i|\geq c\ell$. As the graphs $H(w_i)$ are vertex-disjoint $(2c\ell,2c\ell)$-hubs, we have $|N(\cup_iT_i)|\geq 2c\ell\cdot c\ell\cdot m^3$, and hence, we have
$$|N(\cup_iT_i)\setminus W'|\geq 2c\ell\cdot c\ell\cdot m^3-|W'|\overset{\eqref{Wupbound}}{\geq} c^2\ell^2m^3.$$

At least $\ell m^3-m^3\cdot c\ell\ge \ell m^3/2$ vertices $u_i$ have not been involved in the connections made. Call these vertices $u'_1,\ldots,u'_{p}$, where $p\ge \ell m^3/2$. By~\ref{STEP2}, the set $W'$ is disjoint from $\cup_iB_1(u'_i)$, and note that, as $H(u'_i)$ are disjoint $(2m^2,2c\ell)$-hubs, we have 
$$|\cup_iH(u_i')\setminus W'|\geq 2m^2\cdot 2c\ell\cdot \frac{\ell m^3}{2}-c^2\ell^2m^3 \ge c^2\ell^2m^3.$$

We will apply Lemma~\ref{diameter} to connect $N(\cup_iT_i)\setminus W'$ and $\cup_iH(u_i')\setminus W'$, while avoiding the vertices in $W'$.
Recall that $\ep(x)$ is decreasing, and, since $n/\ell^2\ge K$ is sufficiently large and $s\ge 2$,
\begin{eqnarray}\label{eq-ep}
\ep(n)=\frac{\ep_1}{\log^2(15n/\ep_2\ell^2)}\ge \frac{4}{\log^{3}(n/\ell^2)}\ge\frac{4}{m^{3/4}}.
\end{eqnarray}
Hence, setting $y:=c^2\ell^2 m^3$, we have 
$$\frac{1}{4}\cdot \ep(y)\cdot y\ge \frac{1}{4}\cdot \ep(n)\cdot y\overset{\eqref{eq-ep}}{\ge}\frac{1}{4}\cdot \frac{4}{m}\cdot y=c^2\ell^2m^2.$$
Thus, by~\eqref{Wupbound} and Lemma~\ref{diameter}, there is a path of length at most 
$$\frac{2}{\ep_1}\log^3\left(\frac{15n}{\ep_2\ell^2}\right)+1\le \log^4\left(\frac{n}{\ell^2}\right)\le m$$ 
from some $T_i$ to some $H(u'_j)$ which avoids all the vertices in $W'$. Taking some shortest such path, we may take a $w_i,u'_j$-path of length at most $1+m+2\le 2m$ in $G-W'$ which connects two pairs of vertices unconnected by the process and satisfies~\ref{STEP1} and~\ref{STEP2}, a contradiction.
\end{proof}

By Claim~\ref{claim-unit} we can take a vertex $w_i$ with $c\ell$ vertices $u_j$, $1\le j\le c\ell$, connected to~$w_i$ under the rules~\ref{STEP1} and~\ref{STEP2}. It suffices to find, in each $(2m^2,2c\ell)$-hub $H(u_j)$, an $(m^2, c\ell)$-hub which is disjoint from any vertices used in all the connections with~$w_i$, since such a hub together with all the $w_i,u_j$-paths forms a $\left(c\ell, m^2, c\ell,2m\right)$-unit. By~\ref{STEP2}, at most 1 vertex in $S_1(u_j)$ is used in the connections with $w_i$. At most $c\ell \cdot 2m=2c\ell m$ vertices are used in the connections with $w_i$. Therefore, at most $2m$ vertices $v$ in $S_1(u_j)$ can have more than $c\ell$ vertices from $S_1(v)$ in the connections with $w_i$. We can then take a set of $m^2$ vertices $v$ in $S_1(u_j)$ along with $c\ell$ vertices in each $S_1(v)$ which avoid the connections with $w_i$, allowing us to form the desired $(m^2, c\ell)$-hub.
\end{proof}


\subsection{Constructing subdivisions from units}\label{sec-finish}
Finally in this section, we will expand and connect units to form a subdivision.

\begin{proof}[Proof of Lemma~\ref{lem-dense}] Let~$G$ be an $n$-vertex bipartite, $K_{s,t}$-free $(\ep_1,\ep_2d^{s/(s-1)})$-expander with $\de(G)\ge d/16$ and $\Delta(G)\le d\log^{10s}(n/d^{s/(s-1)})$, where $d\ge \log^{20s}n$. Let $c=1/800t$, $\ell=d^{\frac{1}{2}\frac{s}{s-1}}$, $m=\log^{2s}(n/\ell^2)$ and let $K:=K(\ep_1,\ep_2,s,t)$ be sufficiently large that it has the property in Lemma~\ref{lem-unit} and that, if $n\geq K^{s/(s-1)}$, then $\ep(n)\geq 4/m^{3/4}$ (that is,~\eqref{eq-ep} holds) and $m\ge 8/c$. Supposing $n\geq K^{s/(s-1)}$ then, we may find in $G$ $c\ell/2$ $(c\ell,m^2,c\ell,2m)$-units with disjoint interiors. Let these units be $F_1,\ldots, F_{c\ell/2}$ with core vertices $v_1,\ldots, v_{c\ell/2}$ and denote by $u_{i,j}$ the center of the $j$-th hub in $F_i$, $1\le j\le c\ell$. We will construct a $K_{c\ell/4}$-subdivision, which is sufficient to prove the lemma with the constant $c/4$.

Let $W$ be the union of the vertices in all the $v_i,u_{i,j}$-paths in all the units, including their endvertices, so that $|W|\leq c\ell\cdot (2m+1)\cdot \frac{c\ell}{2}\leq 2c^2\ell^2m$. 

\medskip

\noindent\textbf{Algorithm~R}: Greedily connect pairs of hubs from different units using paths between the sets $S_1(u_{i,j})$ and $S_1(u_{i',j'})$ under the following rules.

\begin{enumerate}[label = \textbf{R\arabic{enumi}}]
\item\label{rule1} Each connection uses a path of length at most $6m$, avoiding vertices used in previous connections and vertices in $W$.
\item\label{rule1b} Each hub is in at most one connection.
\item\label{rule2}  For each pair of units, there is at most one connection between their respective hubs.
\item\label{rule3}  If more than $\ell m$ vertices in the interior of any unit have been used in connections, then discard that unit (along with the hubs it contains).
\end{enumerate}

Note first that by~\ref{rule1} and~\ref{rule2}, the number of vertices used in all the connections is at most $(6 m+1)\cdot {c\ell/2\choose 2}\le c^2\ell^2m$. Thus, as the interior of the units are disjoint, we discard at most $\frac{c^2\ell^2m}{\ell m}\le c\ell/4$ units by~\ref{rule3}.

\begin{claim}\label{claimcomplete} For every pair of remaining units there is a connection between two of their respective hubs.
\end{claim}
\begin{proof}[Proof of Claim~\ref{claimcomplete}] Contrary to the claim, suppose there is a pair of units $F_i$ and $F_j$ with no connection between two of their respective hubs.  By~\ref{rule2}, there are at least $c\ell/2$ hubs in $F_i$ not involved in connections. Say these hubs are $H(u_{i,i'})$, $i'\in I$. Let $A_i$ be the set of vertices in $\cup_{i'\in I}S_1(u_{i,i'})$ not involved in connections. 
 As $F_i$ has not been discarded, by~\ref{rule3} we have 
\[
|A_i|\ge |\cup_{i'\in I}S_1(u_{i,i'})|-\ell m\geq \frac{c\ell}{2}\cdot m^2-\ell m\geq \frac{c\ell m^2}{4}.
\]
As the hubs in $F_i$ are disjoint, we have 
$$|N(A_i)\setminus W|\geq \frac{c\ell m^2}{4}\cdot c\ell-2c^2\ell^2m\geq \frac{c^2\ell^2m^2}{8}:=y.$$
Similarly, if $A_j$ is defined comparably to $A_i$, then $|N(A_j)\setminus W|\ge y$. The number of vertices to avoid in~\ref{rule1} is at most
$$c^2\ell^2m+|W|\le 3 c^2\ell^2m .$$
Using that $\ep(n)\geq 4/m^{3/4}$ and that $\ep(x)$ is decreasing, we have
$$\frac{1}{4}\cdot\ep(y)\cdot y\ge \frac{1}{4}\cdot\ep(n)\cdot y\ge\frac{1}{4}\cdot \frac{4}{m^{3/4}}\cdot y=\frac{c^2\ell^2m^{5/4}}{8}>3c^2\ell^2m.$$
Therefore, by Lemma~\ref{diameter}, there is a path of length at most $m$ between $N(A_i)\setminus W$ and $N(A_j)\setminus W$ that avoids the vertices specified by~\ref{rule1}, a contradiction.
\end{proof}

Finally, we note that we can put the paths we have found together to create a $K_{c\ell/4}$-subdivision. For each pair of remaining core vertices $v_i$ and $v_j$, take the connection between two of their respective hubs along with the paths from the units between the centre vertex of those hubs and the vertices $v_i$ and $v_j$, to create a $v_i,v_j$-path. For each pair of core vertices $v_i$ and $v_j$ we then have a $v_i,v_j$-path and by~\ref{rule1} and \ref{rule1b} these paths are disjoint outside of their endvertices. As there are at least $c\ell/4$ units remaining, we have a $K_{c\ell/4}$-subdivision, as required.
\end{proof}


\section{Constructing subdivisions in a sparse expander}\label{sec-sparse}
In this section, we prove Lemma~\ref{sparsesub}, which covers the case when the subgraph found using Lemma~\ref{max-deg-reduction} is sparse. Here, this means that $d\le \log^{20s}n$. We break its proof into the following two results. First, in Lemma~\ref{lemma-sparse-sub}, we show that if the graph is sparse, then under some mild expansion property (i.e.\ if it is an~$(\ep_1,\ep_2d)$-expander) the graph contains a subdivision of order linear in the minimum degree, even without the $K_{s,t}$-free condition. Then, in Proposition~\ref{lemma-sparse-expansion}, we show that any graph satisfying the hypothesis of Lemma~\ref{sparsesub} has the expansion property required in Lemma~\ref{lemma-sparse-sub} (i.e.~every $K_{s,t}$-free $(\ep_1,\ep_2d^{s/(s-1)})$-expander is an $(\ep_1,\ep_2d)$-expander).

\begin{lemma}\label{lemma-sparse-sub}
For each $0<\ep_1<1$, $0<\ep_2\leq 1/20$, and $s\ge 2$, there exists $c_1=c_1(\ep_1,\ep_2,s)>0$ for which the following holds for each $d>0$ and $n\in \mathbb{N}$ with $d\le\log^{20s}n$. If $G$ is an $n$-vertex, bipartite, $(\ep_1,\ep_2d)$-expander, with $\de(G)\ge d/16$ and $\Delta(G)\le d\log^{10s}n$, then $G$ contains a $K_{c_1d}$-subdivision.
\end{lemma}

\begin{prop}\label{lemma-sparse-expansion}
Let $0<\ep_1<1$, $0<\ep_2<1/10^5t$, and let $t\ge s\ge 2$ be integers. If $G$ is a $K_{s,t}$-free, $(\ep_1,\ep_2d^{s/(s-1)})$-expander with $\de(G)\ge d/16$, then $G$ is also an $(\ep_1,\ep_2d)$-expander.
\end{prop}

Note that together these two results imply Lemma~\ref{sparsesub}.

\subsection{Proof of Lemma~\ref{lemma-sparse-sub}}
We first sketch here the idea of the proof. Using that the graph is sparse and has bounded maximum degree, we first find a collection of vertices $v_i$ that are pairwise far apart and will serve as the core vertices of our clique subdivision (see Proposition~\ref{pickcorners}). We then robustly grow two balls around each vertex $v_i$, one inner ball of medium size called $B^r(v_i)$ (as in Lemma~\ref{corner}) and one outer ball of large size called $B^k(v_i)$ (as in Lemma~\ref{secondexpand}). Due to the fact that all the vertices $v_i$ are pairwise far apart, all the sets $B^k(v_i)$ are pairwise disjoint. To construct the desired clique subdivision, we connect pairs $v_i,v_j$ using a shortest path between the outer balls around them while avoiding all vertices in the inner balls of other core vertices, i.e.~$\cup_{p\ne i,j}B^{r}(v_p)$. Using the robust expansion guaranteed by Lemmas~\ref{corner} and~\ref{secondexpand}, we will be able to grow new inner and outer balls around $v_i$ to enable us to connect $v_i$ to more core vertices.

We start therefore with Proposition~\ref{pickcorners}, which shows there are many vertices which are pairwise far apart in the graph.

\begin{prop}\label{pickcorners} Let $s\geq 1$.
Suppose a graph $G$ has $n$ vertices and maximum degree at most $\log^{30s}n$. For sufficiently large $n$, $G$ contains at least $n^{1/5}$ vertices which are pairwise a distance at least $\log n/(50s\log\log n)$ apart.
\end{prop}
\begin{proof} Let $k=\log n/(50s\log\log n)$.
Let $Y$ be a maximal set of vertices of $G$ which are pairwise a distance at least $k$ apart and suppose, for contradiction, that $|Y|\leq n^{1/5}$. Then
\begin{equation*}
|B^{k}(Y)|\leq 2|Y|(\De(G))^k\le 2n^{1/5}(\log^{30s}n)^{k}\le 2 n^{1/5}\exp\left(\frac{30\log n}{50}\right)<n.
\end{equation*}
Thus there exists some vertex $v\notin B^{k}(Y)$, which must be a distance at least $k$ away from each of the vertices in $Y$, a contradiction.
\end{proof}

In the next lemma we show that if a vertex $v$ has paths leading into it which do not have many vertices near $v$ (see Definition~\ref{csp-defn}) then we can expand out from $v$ while avoiding the interior vertices of these paths. This lemma is a development of techniques used in~\cite{Richard}.

\begin{defn}\label{csp-defn}
We say that paths $P_1,\ldots,P_q$, each starting with the vertex $v$ and contained in the vertex set $W$, are \emph{consecutive shortest paths from $v$ in $W$} if, for each $i$, $1\leq i\leq q$, the path $P_i$ is a shortest path between its endpoints in the set $W-\cup_{j<i}P_j+v$.
\end{defn}

\begin{lemma} \label{corner} 
Let $0<\ep_1<1$, $0<\ep_2<1/20$ and $s\ge 1$. Then there is some $c>0$ and $d_0\in \mathbb{N}$ for which the following holds for any $n$ and $d$ with $d_0\leq d\le \log^{20s}n$. Suppose $H$ is an $n$-vertex $(\ep_1,\ep_2 d)$-expander with $\de(H)\geq d/16$. Letting $r=(\log\log n)^5$ and $P=\cup_iV(P_i)$, if $q\leq cd$ and $P_{1},\ldots,P_{q}$ are consecutive shortest paths from $v$ in $B^{r}(v)$ then
\[
|B^{r}_{H-P+v}(v)|\geq d^2\log^7 n.
\]
\end{lemma}
\begin{proof} We will choose $c:=c(\ep_1,\ep_2)$ sufficiently small and $d_0:=d_0(\ep_1,\ep_2,s)$ sufficiently large later. Note that as $\log^{20s}n\ge d\ge d_0$ we can also make $n$ sufficiently large with respect to $\ep_1$, $\ep_2$ and $s$. 

Let $F=H-P+v$. We will show by induction on $p\geq 1$ that, if $|B^p_{F}(v)|\leq d^2\log^{7} n$, then we have $p<r$ and
\begin{equation}\label{equa1}
|N_{F}(B^p_{F}(v))|\ge \frac{1}{2}|B_F^p(v)|\cdot \ep(|B_F^p(v)|).
\end{equation}
We will also show that $|B_F(v)|\geq d/20$, which together with this inductive statement will prove the lemma. Indeed, if this holds, then for sufficiently large $d\ge d_0$ and hence sufficiently large $n$, we have for each $1\le p<r$ that
\begin{eqnarray*}
|N_{F}(B^p_{F}(v))|&\ge& \frac{1}{2}|B_F^p(v)|\cdot \ep(|B_F^p(v)|)
= \frac{\ep_1|B_F^p(v)|}{2\log^2\left(\frac{15|B^p_{F}(v)|}{\ep_2d}\right)}\ge |B^p_{F}(v)|\cdot\frac{\ep_1}{2\log^2\left(\frac{15d^2\log^7n}{\ep_2d}\right)}\\
&\geq& |B^p_{F}(v)|\cdot\frac{\ep_1}{2\log^2\left(\log^{30s}n\right)}\ge\frac{|B_F^p(v)|}{(\log\log n)^3},
\end{eqnarray*}
where we have used that $|B_F^p(v)|\geq |B_F(v)|\geq d/20\geq \ep_2 d/2$ to apply the expansion property. Thus, we have
$$|B^r_{F}(v)|>\left(1+\frac{1}{(\log\log n)^3}\right)^{r-1}\geq \exp\left(\frac{r-1}{2(\log\log n)^3}\right)\gg \log^{50s}n\ge d^2\log^7n.$$

Thus we need only prove the inductive statement holds and $|B_F(v)|\geq d/20$. Observe that, if $0\leq p<r$, then, as the paths $P_i$ are consecutive shortest paths from~$v$ in $B^{r}(v)$, only the first $p+2$ vertices of each path $P_i$, including~$v$, can belong in $N_H(B^{p}_{H-\cup_{j<i}(V(P_j)\setminus\{v\})}(v))$. Therefore, if $p<r$, as $F=H-P+v$, only the first $p+2$ vertices of each of the paths~$P_i$, including the vertex~$v$, can belong in $N_H(B^{p}_{F}(v))$. Hence, as we have at most $cd$ paths $P_i$, if $p<r$, then $|N_H(B^{p}_{F}(v))\cap (P\setminus\{v\})|\leq (p+1)cd$, so that
\begin{equation}\label{equa2}
|N_H(B^{p}_{F}(v))\setminus N_{F}(B^{p}_{F}(v))|\leq (p+1)cd.
\end{equation}
Thus, if $|B^p_{F}(v)|\leq d^2\log^7n$, then, once we have shown that $p<r$, we can use the inequality in~\eqref{equa2}.

In particular, when $p=0$, combining~\eqref{equa2} and the minimum degree condition for~$H$ implies that $|B_F(v)|\geq |N_{F}(v)|\geq d/16-cd\ge d/20$.

We first verify the base case of~\eqref{equa1} when $p=1$. Since $|B_F(v)|\ge d/20\ge \ep_2d$ and that $\ep(x)\cdot x$ is increasing when $x\ge \ep_2d$, we have that, for sufficiently small $c$,
$$2cd\le \frac{1}{2}\cdot \frac{\ep_1}{\log^2(3/4\ep_2)}\cdot\frac{d}{20}=\frac{1}{2}\cdot \ep\left(\frac{d}{20}\right)\cdot \frac{d}{20}\le \frac{1}{2}\ep(|B_{F}(v)|)|B_{F}(v)|.$$
The base case then follows from the expansion property, as
\begin{eqnarray*}
|N_{F}(B_{F}(v))|&\overset{\eqref{equa2}}{\ge}& |N_{H}(B_{F}(v))|-2cd\ge \ep(|B_{F}(v)|)|B_{F}(v)|-2cd\ge \frac{1}{2}\ep(|B_{F}(v)|)|B_{F}(v)|.
\end{eqnarray*}

%

Now, suppose that $p\geq 2$, $|B^p_{F}(v)|\leq d^{2}\log^{7} n$, and that the induction hypothesis holds for all $1\leq p'<p$. Let $\alpha$ be defined by $|B^p_{F}(v)|=\al \ep_2 d/15$, and note that $\al\ge 3$. Then
\begin{eqnarray}\label{eq-pball}
\ep(|B^p_{F}(v)|)=\frac{\ep_1}{\log^2\left(\frac{15}{\ep_2d}\cdot\frac{\al\ep_2d}{15}\right)}=\frac{\ep_1}{\log^2\al}.
\end{eqnarray}

The induction hypothesis for each $p'$, $1\leq p'<p$, limits how large $p$ can be. The size of the ball $B^p_{F}(v)$ has increased by at most a factor of $\frac{|B_F^p(v)|}{|B_F(v)|}\le\frac{\al\ep_2d/15}{d/20}\le \al$ from $|B_F(v)|$. On the other hand, by the induction hypothesis and~\eqref{eq-pball}, and as $\ep(x)$ is decreasing in $x$ when $x\ge \ep_2d$, at each increase in radius the size of $B^p_F(v)$ increases by at least a factor of $\left(1+\ep_1/2\log^2\al\right)$. Thus, $\left(1+\ep_1/2\log^2\al\right)^{p-1}\le \al$, so that
\begin{equation}\label{equa3}
p-1\leq \frac{\log \al}{\log\left(1+\ep_1/2\log^2\al\right)}\leq \frac{4\log^3\al}{\ep_1},
\end{equation}
where we have used that $\log (1+x)\ge x/2$ when $x$ is small. As for sufficiently large $d\ge d_0$, and hence sufficiently large $n$, we have
$$\al=\frac{15|B_F^p(v)|}{\ep_2d}\le\frac{15d^2\log^7n}{\ep_2d}\le \log^{30s} n,$$ 
by~\eqref{equa3} we have that $p-1\leq 4(\log(\log^{30s} n))^3/\ep_1\ll r/2$, and hence $p< r$. Note that when $\al\geq 3$ the function $\al\mapsto \log^5\al/\al$ is bounded above by some universal constant, $K$ say. Therefore, for sufficiently large $d$, we have
\begin{eqnarray}
(p+1)cd&\overset{\eqref{equa3}}{\leq}& \frac{8\log^3\al}{\ep_1}\cdot cd
\leq \frac{8cd}{\ep_1}\cdot \frac{K\cdot \al}{\log^2\al}=\frac{120cK}{\ep_1\ep_2}\cdot \frac{\al\ep_2 d}{15\log^2\al}\nonumber\\
&\overset{\eqref{eq-pball}}{=}&
\frac{120cK}{\ep_1^2\ep_2}\cdot \ep(|B^p_{F}(v)|)\cdot|B^p_{F}(v)|
\leq\frac{1}{2}\ep(|B^p_{F}(v)|)\cdot|B^p_{F}(v)|,\label{equa4}
\end{eqnarray}
for $c$ sufficiently small depending on $\ep_1$, $\ep_2$ and the universal constant $K$.
By~\eqref{equa2}, the expansion property, and~\eqref{equa4}, we have
\begin{eqnarray*}
|N_{F}(B^p_{F}(v))|&\ge& |N_{H}(B^p_{F}(v))|-(p+1)cd\ge \ep(|B^p_{F}(v)|)|B^p_{F}(v)|-(p+1)cd
\\
&\ge&
\frac{1}{2}\ep(|B^p_{F}(v)|)|B^p_{F}(v)|.
\end{eqnarray*}
Therefore, \eqref{equa1} holds for $p$, and thus the inductive hypothesis for $p$ holds.
\end{proof}

Having expanded around a vertex while avoiding the paths leading into it, we will expand again while avoiding more vertices. Proposition~\ref{secondexpand} confirms this will result in a large set.

\begin{prop}\label{secondexpand} For each $s\geq 1$, $0<\ep_1<1$ and $\ep_2>0$, there exists $d_0:=d_0(\ep_1,\ep_2,s)$ so that the following is true for each $d\ge d_0$. Suppose that $H$ is an $(\ep_1,\ep_2 d)$-expander with $n$ vertices and let $k=\log n/100s\log\log n$. If $Y,W\subset V(G)$ are disjoint sets with $|Y|\geq d^2\log^7 n$ and $|W|\leq d^2\log^4 n$, then $|B^k_{G-W}(Y)|\geq \exp(\sqrt[4]{\log n})$.

\end{prop}
\begin{proof} Note that $n\ge |Y|\ge d$, and thus by requiring $d\ge d_0$ to be large, we can make $n$ large.
For any $p\geq 0$, if $\ep_2d/5\leq |B^p_{G-W}(Y)|\leq \exp(\sqrt[4]{\log n})$ then
\[
\ep(|B^p_{G-W}(Y)|)\geq \frac{\ep_1}{\sqrt{\log n}}.
\]
Thus, as $|B^p_{G-W}(Y)|\geq |Y|\geq d^2\log^7n$, we have $\ep(|B^p_{G-W}(Y)|)|B^p_{G-W}(Y)|\geq 2|W|$. Therefore, if $|B^p_{G-W}(Y)|\leq \exp(\sqrt[4]{\log n})$, then, as $H$ is an $(\ep_1,\ep_2 d)$-expander,
\begin{equation*}
|N_{G-W}(B^p_{G-W}(Y))|\geq \frac{\ep_1}{2\sqrt{\log n}}|B^p_{G-W}(Y)|.
\end{equation*}
We must have that $|B^k_{G-W}(Y)|< \exp(\sqrt[4]{\log n})$, for otherwise we are done, whereupon
\[
|B^k_{G-W}(Y)|\geq \left(1+\frac{\ep_1}{2\sqrt{\log n}}\right)^k\geq \exp\left(\frac{\ep_1k}{4\sqrt{\log n}}\right)
\geq \exp\left(\sqrt[4]{\log n}\right),
\]
a contradiction.
\end{proof}

Finally, as sketched at the start of this section, we can prove Lemma~\ref{lemma-sparse-sub}.

\begin{proof}[Proof of Lemma~\ref{lemma-sparse-sub}] 
Define parameters as follows:
$$r=(\log\log n)^5,\quad \quad k=\frac{\log n}{100s\log \log n},\quad\quad c_1=\min\{c(\ep_1,\ep_2),1/d_0(\ep_1,\ep_2)\},$$ 
where $c$ and $d_0$ come from Lemma~\ref{corner} and Proposition~\ref{secondexpand} respectively. Note that as $\delta(G)\geq d/16>0$,~$G$ contains a $K_1$-subdivision, and hence we may assume that $d\geq d_0$, for otherwise we are done.
Since $\De(G)\le d\log^{10s}n\le\log^{30s}n$, by Proposition~\ref{pickcorners}, we may find in $G$ vertices $v_1,\ldots,v_{c_1d}$ which are pairwise a distance at least $2k$ apart.

Let $I\subset [c_1d]^{(2)}$ be a maximal subset for which we can find paths $Q_e$, $e\in I$, so that the following hold.
\begin{enumerate}[label = \rm{(\roman{enumi})}]
\item For each $ij\in I$, $Q_{ij}$ is a $v_i,v_j$-path with length at most $2\log^4 n$. \label{AA1}
\item For each $e,e'\in I$, the paths $Q_e$ and $Q_{e'}$ are disjoint except for, potentially, their end vertices.\label{AA2}
\item For each $i$, there is some ordering of $\{Q_e[B^r(v_i)]:e\in I, i\in e\}$ such that they are consecutive shortest paths from $v_i$ in $B^r(v_i)$.\label{AA3}
\item For each $e\in I$ and $i\notin e$, $B^r(v_i)$ and $Q_e$ are disjoint.\label{AA4}
\end{enumerate}

If $I=[c_1d]^{(2)}$, then by~\ref{AA2} the paths $Q_e$, $e\in I$, form a $K_{c_1d}$-subdivision as required. Thus, suppose there is some $ij\in [c_1d]^{(2)}\setminus I$. Let $W=(\cup_{e\in I}V(Q_e))\setminus\{v_i,v_j\}$. By \ref{AA3}, \ref{AA4} and Lemma~\ref{corner}, we have $|B^r_{G-W}(v_i)|\geq d^2\log^7 n$. By Proposition~\ref{secondexpand}, and as $|W|\leq d^2\log^4 n$ (due to~\ref{AA1}), we have $|B^{k+r}_{G-W}(v_i)|\geq x:=\exp(\sqrt[4]{\log n})$.

Let $W'=W\cup (\cup_{i'\notin \{i,j\}} B^r(v_{i'}))$, i.e.~$W'$ is the set of all vertices that are either used in some connection or are in some inner ball of other core vertices. As we chose the vertices $v_{i'}$ to be pairwise at least a distance $2k>k+2r$ apart, both $B^{k+r}_{G-W}(v_i)$ and $B^{k+r}_{G-W}(v_j)$ are disjoint from $W'$. Therefore, we have that $|B^{k+r}_{G-W'}(v_i)|,|B^{k+r}_{G-W'}(v_j)|\geq x$. As $\Delta(G)\leq \log^{30s} n$, we then have
\begin{equation*}
|W'|\leq |W|+c_1d\cdot 2(\log^{30s}n)^r\le \log^{40sr}n= \exp(40s(\log\log n)^6)\ll \ep(x)x/4.
\end{equation*}
Therefore, by Lemma~\ref{diameter}, there is a $B^{k+r}_{G-W'}(v_i),B^{k+r}_{G-W'}(v_j)$-path in $G-W'$ with length at most $\log^4 n$. Thus, if we let $Q_{ij}$ be a shortest $v_i,v_j$-path in $G-W'$, then $Q_{ij}$ has length at most $\log^4n+2k+2r\le2\log^4 n$ in $G-W'$. The paths $Q_e$, $e\in I\cup \{ij\}$ satisfy the conditions \ref{AA1}--\ref{AA4} above, contradicting the choice of $I$.
\end{proof}

\subsection{Proof of Proposition~\ref{lemma-sparse-expansion}}
It is left then only to prove Proposition~\ref{lemma-sparse-expansion}.
\begin{proof}[Proof of Proposition~\ref{lemma-sparse-expansion}] Let $0<\ep_1<1$ and $0<\ep_2\le 1/10^5t$. We need to show that for every set $X\subset V(G)$ with $\ep_2d/2\le |X|\le n/2$, we have $|N(X)|\ge \ep(|X|,\ep_1,\ep_2d)\cdot|X|$. Recall that $\ep(x,\ep_1,k)$ is an increasing function in $k$ and $G$ is an $(\ep_1,\ep_2d^{s/(s-1)})$-expander. Thus, for every set $X\subset V(G)$ of size $x\ge \frac{1}{2}\ep_2d^{s/(s-1)}$ with $x\leq n/2$, 
$$|N(X)|\ge \ep(x,\ep_1,\ep_2d^{s/(s-1)})\cdot x\ge \ep(x,\ep_1,\ep_2d)\cdot x.$$
Fix now a set $X\subset V(G)$ of size $x$ with $\frac{1}{2}\ep_2d\le x\le \frac{1}{2}\ep_2d^{s/(s-1)}$.
\begin{claim}\label{claim-exp}
$|N(X)|\geq |X|$.
\end{claim}
\begin{proof}[Proof of Claim~\ref{claim-exp}]
As $G$ is bipartite, we can pick a subset $X'\subset X$ in one class of the partition with $|X'|\geq |X|/2$. If there is no vertex $v$ with $d/64$ neighbours in $X'$, then, as $\delta(G)\geq d/16$, $|N(X)|\geq (d/16)|X'|/(d/64)=4|X'|$, and thus $|N(X)|\geq |N(X')|-|X|\geq |X|$. Therefore, we may assume there is a vertex $v$ with at least $d/64$ neighbours in $X'$. Letting $A=N(v)\cap X'$ and $B=N(X')\setminus \{v\}$, there is no copy of $K_{s-1,t}$ with $s-1$ vertices in $B$ and $t$ vertices in $A$. Therefore, by Corollary~\ref{cor-KST} we have
\[
|N(X')|\geq |B|\ge (d/16-1)(d/64)^{1/(s-1)}/et\geq \ep_2 d^{s/(s-1)}\geq 2|X|,
\]
and thus $|N(X)|\ge |N(X')|-|X|\geq |X|$.
\end{proof}
Therefore, combining this claim with the fact that $\ep(x,\ep_1,\ep_2 d)$ is a decreasing function in $x$,
$$|N(X)|\ge x>\frac{\ep_1}{\log^2(15/2)}\cdot x=\ep\left(\frac{\ep_2d}{2},\ep_1,\ep_2d\right)\cdot x\ge \ep\left(x,\ep_1,\ep_2d\right)\cdot x,$$
as desired.
\end{proof}
\section{Concluding remarks}\label{sec-conclude}
We have shown that every $K_{s,t}$-free graph with average degree $d$ has a subdivision of a clique of order $\Omega(d^{\frac{1}{2}\frac{s}{s-1}})$. When $s=2$, this settles in a strong sense a conjecture of Mader that every $C_4$-free graph contains a clique subdivision with order linear in average degree of the parent graph. As discussed in the introduction, in general the $K_{s,t}$-free condition seems like the most natural to force a subdivision larger than that guaranteed in a general graph with average degree $d$. It would be interesting to generalise Theorem~\ref{mainC4} to non-complete bipartite forbidden subgraphs. Here our methods are limited as we do not have a comparable version of Lemma~\ref{lem-KST}. Note that forbidding 3-colourable graphs does not force a larger clique subdivision as the extremal examples for $s(d)$ mentioned in the introduction are bipartite.

\section*{Acknowledgements}
This project was started while the second author was visiting the University of Birmingham, and the authors would like to thank the university for its hospitality.


\end{document}